\documentclass[letterpaper, 10 pt, conference]{ieeeconf}  % Comment this line out if you need a4paper

\IEEEoverridecommandlockouts                              % This command is only needed if 
\usepackage{amsmath,amssymb,bm,bbm,mathrsfs,amscd}
\usepackage{color}
\usepackage{dsfont}
\usepackage{graphicx}
\usepackage{epsfig}
\usepackage{subfigure}
\usepackage{algorithm}
\usepackage[noend]{algpseudocode}
\usepackage{tikz}

\newtheorem{theorem}{Theorem}

\newtheorem{lemma}{Lemma}

\newtheorem{assumption}{Assumption}
\newtheorem{standassumption}{Standing Assumption}

\newcommand{\RR}{{\mathbb{R}}}
\newcommand{\NN}{{\mathbb{N}}}
\newcommand{\EE}{{\mathbb{E}}}
\newcommand{\PP}{{\mathbb{P}}}
\newcommand{\JJ}{{\mathbb{J}}}
\newcommand{\FF}{{\mathbb{F}}}

\newcommand{\SVI}{{\mathrm{SVI}}}
\newcommand{\mc}{\mathcal}
\newcommand{\norm}[1]{\left\|#1\right\|}
\newcommand{\normsq}[1]{\left\|#1\right\|^2}
\newcommand{\EEk}[1]{\EE\left[#1|\mc F_k\right]}

\newcommand{\op}{\operatorname}

\newcommand{\bs}{\boldsymbol}
\newcommand{\fineass}{\hfill\small$\blacksquare$}

\title{\LARGE \bf
A damped forward--backward algorithm for stochastic generalized Nash equilibrium seeking
}

\author{Barbara Franci$^{1}$ and Sergio Grammatico$^{1}$% <-this % stops a space
\thanks{$^{1}$The authors are with the Delft Center for System and Control, TU Delft, The Netherlands
        {\tt\footnotesize \{b.franci-1, s.grammatico\}@tudelft.nl}}%
\thanks{This work was partially supported by NWO under research projects OMEGA (613.001.702) and P2P-TALES (647.003.003), and by the ERC under research project COSMOS (802348).}}

\begin{document}

\maketitle
\thispagestyle{empty}
\pagestyle{empty}

\begin{abstract}

We consider a stochastic generalized Nash equilibrium problem (GNEP) with expected--value cost functions. 
%Our main contribution is to design a distribute forward--backward algorithm through the primal dual analysis of the Karush--Kuhn--Tucker inclusions. 
Inspired by Yi and Pavel (Automatica, 2019), we propose a distributed GNE seeking algorithm by exploiting the forward--backward operator splitting and a suitable preconditioning matrix. 
Specifically, we apply this method to the stochastic GNEP, where, at each iteration, the expected value of the pseudo--gradient is approximated via a number of random samples.
Our main contribution is to show almost sure convergence of our proposed algorithm if the sample size grows large enough.

%Specifically, we build up this method to the stochastic case and prove convergence of the proposed algorithm to a generalized Nash equilibrium.

\end{abstract}

\section{Introduction}

%1
Generalized Nash equilibrium problems (GNEPs) have been widely studied in the literature. Many results are present concerning algorithms and methodologies to find an equilibrium \cite{facchinei2010,facchinei2007} (and the references therein). The reason for this interest is related with the potential applications ranging from economics to engineering via operation research \cite{pavel2007,kulkarni2012}.
%5
In a GNEP, each agent aims at minimizing his own cost function, under some feasibility constraints. The main feature is that both the cost function and the constraints depend on the stategy chosen by the other agents. Due to the presence of shared constraints, the search for a generalized Nash equilibrium is a challenging task in general. 
%3A 

In the deterministic case, several algorithms are available for finding a GNE, both distributed and semi-decentralized \cite{facchinei2010,belgioioso2018,yi2019,belgioioso2017}.
%6 
Among the available methods for GNEP, an elegant approach is to recast the problem as a monotone inclusion through the use of interdependent Karush--Kuhn--Tucker (KKT) conditions. The resulting Lagrangian formulation allows one to seek equilibrium points as the solutions of an associated variational inequality (VI) that is usually more tractable and the literature is wider than the GNEPs literature \cite{bau2011,facchinei2007} (and reference therein). Equilibria obtained in this way are called variational equilibria.
%we seek for an algorithm for finding zeros of a monotone operator through the use of the Karush--Kuhn--Tucker (KKT) condition.  For this reason, we rely on a class of GNEP that can be solved using a stochastic variational inequality (SVI), which is usually more tractable.
%8
%The connection between SGNEP and SVI is made precisely through the KKT conditions of both the VI and the GNEP, and the consequent primal-dual analysis. 

Given a monotone inclusion problem, a powerful procedure is to use an operator splitting scheme to recast the problem as the search for the zeros of a sum of two monotone operators. One of the simplest schemes is the forward--backward (FB) operator splitting \cite{bau2011}. % that allows to compute the variational equilibrium via fixed--point iterations where the two operators are employed in separate steps \cite{bau2011}.
Convergence of such a scheme is guaranteed if the pseudo-gradient mapping of the game is strongly monotone.
Since we study a game-theoretic problem, the algorithm should be distributed, in the sense that each agent should only know his local cost function and his local constraints.
%9
Unfortunately, the FB splitting does not lead to a distributed algorithm when applied to a GNEP. In the deterministic case, preconditioning has been proposed in \cite{belgioioso2018,yi2019} to overcome this problem. 
To the best of our knowledge, the preconditioned FB operator splitting has not been exploited for stochastic GNEPs (SGNEPs) but it would be very relevant from an algorithmic perspective.
%\tc{blue}{In this paper, we extended this method to stochastic GNEPs (SGNEPs). To the best of our knowledge, we are the first to adapt a preconditioned forward--backward operator splitting for SGNEPs. }

%\tc{blue}{To guarantee that each agent is able to compute his cost function, the information on the decision of the other agents are shared through the influence graph. Moreover, since variational equilibria require consensus of the dual variables \cite{facchinei2007vi}, information on the dual variables are shared through the dual variables graph. In this way, each decision maker is able to control his decision variable and a copy of the dual variables and is able to compute his cost function and his constraints.}
%2 
The literature on stochastic Nash equilibrium problems is less rich than its deterministic counterpart \cite{yu2017,ravat2011,xu2013}. However, several problems of interest cannot be modelled without uncertainty. Among others, consider transportation systems, where one source of uncertainty is due to the drivers perception of travel time \cite{watling2006}; %market competition where firms make a decision on their future supply without knowing the demand \cite{demiguel2009}; 
electricity markets where generators produce energy without knowing the actual demand \cite{henrion2007}; or, more generally, problems that can be modelled as networked Cournot games with market capacity constraints where the demand is uncertain \cite{demiguel2009,abada2013}.  
%3B
Mathematically speaking, a SGNEP is a GNEP where the cost functions of the agents are expected--value functions and the distribution of the random variables is unknown. Existence (and uniqueness) of equilibria is studied in \cite{ravat2011} but the study of convergent algorithms is not fully developed yet \cite{xu2013,yu2017}. 

One possible motivation for this lack of results is the presence of the expected--value cost functions.
When the probability distribution of the random variable is known, the expected value formulation can be solved with a technique for deterministic VI.
However, the pseudo gradient is usually not directly accessible, for instance due to the intractability of computing the expected value. For this reason, in many cases, the solution of a stochastic VI relies on samples of the random variable.
There are, in fact, two main methodologies available: sample average approximation (SAA) and stochastic approximation (SA). The SAA approach replaces the expected value with the average over a huge number of samples of the random variable.
This approach is practical in Monte Carlo Simulations or machine learning, when there is a huge number of data available \cite{staudigl2019, iusem2017}. 
In the SA approach, the decision maker draws only one sample of the random variable. This approach is less computationally expensive and more appropriate in a decentralized framework but, in general, it requires stronger assumptions of the problem data \cite{koshal2013,yousefian2017,yousefian2014}.

In this paper, we formulate a distributed stochastic FB algorithm through preconditioning and prove its consequent convergence.
The associated SVI is obtained in the same way as the deterministic case, i.e., via augmented KKT inclusions. Among the possible approaches for solving an SVI \cite{koshal2013,iusem2017,yousefian2017}, we propose a damped forward--backward scheme \cite{rosasco2016} and prove convergence under proper assumptions, i.e., strong monotonicity of the pseudo gradient game mapping and measurability of the random variable.

\textit{Notation:} We use Standing Assumptions to state technical conditions that implicitly hold throughout the paper while Assumptions are postulated only when explicitly used.

$\langle\cdot,\cdot\rangle:\RR^n\times\RR^n\to\RR$ denotes the standard inner product and $\norm{\cdot}$ is the associated euclidean norm. We indicate that a matrix $A$ is positive definite, i.e., $x^\top Ax>0$, with $A\succ0$. $A\otimes B$ indicates the Kronecker product between matrices $A$ and $B$. Given a symmetric $\Phi\succ0$, the $\Phi$-induced inner product is $\langle x, y\rangle_{\Phi}=\langle \Phi x, y\rangle$. The associated $\Phi$-induced norm is defined as $\norm{x}_{\Phi}=\sqrt{\langle \Phi x, x\rangle}$. ${\bf{0}}_m$ indicates the vector with $m$ entries all equal to $0$. Given $x_{1}, \ldots, x_{N} \in \RR^{n}, \boldsymbol{x} :=\op{col}\left(x_{1}, \dots, x_{N}\right)=\left[x_{1}^{\top}, \dots, x_{N}^{\top}\right]^{\top}$. $\op{J}_F=(\op{Id}+F)^{-1}$ is the resolvent of the operator $F$ where $\op{Id}$ is the identity operator. For a closed set $C \subseteq \RR^{n},$ the mapping $\op{proj}_{C} : \RR^{n} \to C$ denotes the projection onto $C$, i.e., $\op{proj}_{C}(x)=\op{argmin}_{y \in C}\|y-x\|$. $\iota_C$ is the indicator function of the set C, that is, $\iota_C(x)=1$ if $x\in C$ and $\iota_C(x)=0$ otherwise. The set-valued mapping $\mathrm{N}_{C} : \RR^{n} \to \RR^{n}$ denotes the normal cone operator of the set $C$ , i.e., $\mathrm{N}_{C}(x)=\varnothing$ if $x \notin C,\left\{v \in \RR^{n} | \sup _{z \in C} v^{\top}(z-x) \leq 0\right\}$ otherwise. Given $\psi : \RR^{n} \to \RR$, $\op{dom}(\psi) :=\{x \in \RR^{n} | \psi(x)<\infty\}$ is the domain and the subdifferential is the set-valued mapping $\partial \psi(x) :=\{v \in \RR^{n} | \psi(z) \geq \psi(x)+v^{\top}(z-x) \forall z \in \operatorname{dom}(\psi)\}$. 

\section{Generalized Nash equilibrium problem}

\subsection{Equilibrium Problem setup}
We consider a set $\mc I=\{1,\dots,N\}$ of agents, each of them deciding on its action variable $x_i\in\RR^{n_i}$ from its local decision set $\Omega_i\in\RR^{n_i}$ with the aim of minimizing its local cost function. Let $\boldsymbol x_{-i}=\op{col}((x_j)_{j \neq i})$ be the decisions of all the agents except for $i$ and define $n=\sum_{i=1}^N n_i$. We consider that there is some uncertainty in the cost function, expressed through the random variable $\xi:\Xi\to\RR^d$, where $(\Xi, \mc F, \PP)$ is the associated probability space. Then, for each agent $i\in\mc I$, we define the cost function $\JJ_i: \RR^{n} \rightarrow \RR$ as 
\vspace{-.15cm}\begin{equation}
\JJ_i(x_i,\boldsymbol x_{-i}):=\EE_\xi[J_i(x_i,\boldsymbol x_{-i},\xi(\omega))]
\vspace{-.15cm}\end{equation}
for some measurable function $J_i:\mc \RR^{n}\times \RR^d\to \RR$. We note that the cost function depends on the local variable $x_i$, the collective decision of the other agents $\bs x_{-i}$ (or a subset) and the random variable $\xi(\omega)$. $\EE_\xi$ represents the mathematical expectation with respect to the distribution of the random variable $\xi$\footnote{For brevity, we use $\xi$ instead of $\xi(\omega)$, $\omega\in\Xi$, and $\EE$ instead of $\EE_\xi$.}. We assume that $\EE[J_i(\boldsymbol x,\xi)]$ is well defined for all the feasible $\boldsymbol x=\op{col}(x_1,\dots,x_N)$.
Furthermore, we consider a game with affine shared constraints, $A\bs x\leq b$. Thus we denote the feasible decision set of each agent $i \in \mc I$ by the set-valued mapping 
\vspace{-.15cm}\begin{equation}\label{constr}
\mc{X}_i(\boldsymbol x_{-i}) :=\left\{y_i \in \Omega_i\; | \;\textstyle{A_i y_i \leq b-\sum_{j \neq i}^{N} A_j x_j}\right\},
\vspace{-.15cm}\end{equation}
where $A_i \in \RR^{m \times n_i}$ and $b\in\RR^m$. The matrix $A_i$ defines how agent $i$ is involved in the coupling constraints. The collective feasible set can be then written as
\vspace{-.15cm}\begin{equation}\label{collective_set}
\bs{\mc{X}}=\bs\Omega \cap\left\{\boldsymbol{y} \in \RR^{n} | A\boldsymbol{y}-b \leq {\bf{0}}_{m}\right\}
\vspace{-.15cm}\end{equation}
where $\bs\Omega=\prod_{i=1}^N\Omega_i$ and $A=\left[A_{1}, \ldots, A_{N}\right]\in\RR^{m\times n}$.
Note that there is no uncertainty in the constraints.

Next, we postulate standard assumptions for the cost functions and the constraints set.
\begin{standassumption}\label{ass_cost}
For each $i \in \mc I$ and $\boldsymbol{x}_{-i} \in \mc{X}_{-i}$ the function $\JJ_{i}(\cdot, \boldsymbol{x}_{-i})$ is convex and continuously differentiable. \fineass
\end{standassumption}

\begin{standassumption}\label{ass_X}
For each $i \in \mc I,$ the set $\Omega_{i}$ is nonempty, compact and convex.
The set $\mc{X}$ satisfies Slater's constraint qualification. \fineass
\end{standassumption}

Formally, the aim of each agent $i$, given the decision variables of the other agents $\bs x_{-i}$, is to choose a action variable $x_i$, that solves its local optimization problem, i.e.,
\vspace{-.15cm}\begin{equation}\label{game}
\forall i \in \mc I:\quad\left\{\begin{array}{cl}
\min_{x_i \in \Omega_i} & \JJ_i(x_i,\boldsymbol x_{-i}) \\ 
\text { s.t. } & A_i x_i \leq b-\sum_{j \neq i}^{N} A_j x_j.
\end{array}\right.
\vspace{-.15cm}\end{equation}
From a game-theoretic perspective, we aim at computing a stochastic generalized Nash equilibrium (SGNE) \cite{ravat2011}, i.e., a collective variable $\boldsymbol{x}^*\in\mc X$ such that, for all $i \in \mc I$:\vspace{-.2cm}
$$\JJ_i(x_i^{*}, \boldsymbol x_{-i}^{*}) \leq \inf \{\JJ_i(y, \boldsymbol x_{-i}^{*})\; | \; y \in \mc{X}_i(\boldsymbol x_{-i}^{*})\}.\vspace{-.2cm}$$
In other words, a SGNE is a set of strategies where no agent can decrease its objective function by unilaterally changing its decision variable.
To guarantee existence of a stochastic equilibrium, let us introduce further assumptions on the local cost functions $J_i$.
\begin{standassumption}\label{ass_J}
For each $i\in\mc I$ and $\xi \in \Xi$, the function $J_{i}(\cdot,\boldsymbol x_{-i},\xi)$ is convex, Lipschitz continuous, and continuously differentiable. The function $J_{i}(x_i,\bs x_{-i},\cdot)$ is measurable and for each $\boldsymbol x_{-i}$, the Lipschitz constant $\ell_i(\boldsymbol x_{-i},\xi)$
is integrable in $\xi .$ \fineass
\end{standassumption}

While, under Standing Assumptions \ref{ass_cost}, \ref{ass_X} and \ref{ass_J}, existence of a GNE of the game is guaranteed by \cite[\S 3.1]{ravat2011}, uniqueness does not hold in general \cite[\S 3.2]{ravat2011}.

Within all possible Nash equilibria, we focus on those that corresponds to the solutions of an appropriate (stochastic) variational inequality. 
Let 
%\vspace{-.15cm}\begin{equation}\label{map_vi}
%\FF(\boldsymbol x)=\op{col}\left(\EE[\nabla_{x_{1}} J_{1}(x_{1}, \boldsymbol x_{-1})], \dots, \EE[\nabla_{x_{N}} J_{N}(x_{N}, \boldsymbol x_{-N})]\right).
%\vspace{-.15cm}\end{equation} 
\vspace{-.15cm}\begin{equation}\label{map_vi}
\FF(\boldsymbol x)=\op{col}\left(\EE[\nabla_{x_{i}} J_{i}(x_{i}, \boldsymbol x_{-i})]_{i\in\mc I}\right).
\vspace{-.15cm}\end{equation} 
%Consider a closed convex set $\mc X \subseteq \RR^{n}$ and a mapping $\FF :\mc X \to \RR^{n}$. 
Formally, the stochastic variational inequality problem $\op{SVI}(\bs{\mc X},\FF)$ is the problem of finding $\bs x^{*} \in \bs{\mc X}$ such that 
\vspace{-.15cm}\begin{equation}\label{eq_vi}
\langle \FF(\bs x^*),\bs x-\bs x^*\rangle\geq 0, \;\text { for all } \bs x \in \bs{\mc X}.
\vspace{-.15cm}\end{equation}
with $\FF(\bs x)$ as in \eqref{map_vi}.
%In a stochastic Nash equilibrium problem, $\FF$ is the pseudo gradient of the cost functions, i.e. 
%\vspace{-.15cm}\begin{equation}\label{map_vi}
%\FF(\boldsymbol x)=\op{col}\left(\EE[\partial_{x_{1}} J_{1}(x_{1}, \boldsymbol x_{-1})], \dots, \EE[\partial_{x_{N}} J_{N}(x_{N}, \boldsymbol x_{-N})]\right).
%\vspace{-.15cm}\end{equation} 
We note that we can exchange the expected value and the gradient in \eqref{map_vi} thanks to Standing Assumption \ref{ass_J} \cite[Lem. 3.4]{ravat2011}. 
We also note that any solution of $\op{SVI}(\mc X , \FF)$ is a generalized Nash equilibrium of the game in (\ref{game}) while the opposite does not hold in general. 
In fact, a game may have a Nash equilibrium while the corresponding VI may have no solution \cite[Prop. 12.7]{palomar2010}.

%For SGNEP, due to the presence of shared constraints, the analysis pass through the dual problem and a forward - backward splitting. Then the problem can be casted as a variational problem. The details of this analysis are described in the forthcoming sections.\\
A sufficient condition for the variational problem in \eqref{eq_vi} to have a solution is that $\FF$ is strongly monotone \cite[Th. 2.3.3]{facchinei2007}, \cite[Lemma 3.3]{ravat2011}, as we postulate next. %For this reason, we take F strongly monotone.
\begin{standassumption}\label{ass_F}
$\FF$ is $\eta$-strongly monotone, i.e., for $\eta>0$, $\langle \FF(x)-\FF(y),x-y\rangle \geq \eta\normsq{x-y}$ for all $x, y \in \RR^{n}$ and $\ell$-Lipschitz continuous, i.e., for $\ell>0$, $\norm{\FF(x)-\FF(y)} \leq \ell\norm{x-y}$ for all $x, y \in \RR^{n}$.
\fineass
\end{standassumption}

\subsection{Operator-theoretic characterization}
In this subsection, we recast the GNEP into a monotone inclusion, namely, the problem of finding a zero of a set-valued monotone operator.

First, we characterize the SGNE of the game in terms of the Karush--Kuhn--Tucker (KKT) conditions for the coupled optimization problems in (\ref{game}). For each agent $i\in\mc I$, let us introduce the Lagrangian function $\mc L_i\left(\boldsymbol x, \lambda_i\right) :=\JJ_i(x_i,\boldsymbol x_{-i})+\iota_{\Omega_i}\left(x_i\right)+\lambda_i^{\top}(A \boldsymbol x-b)$, where $\lambda_i \in \RR_{ \geq 0}^{m}$ is the dual variable associated with the coupling constraints. We recall that the set of strategies $\bs x^{*}$ is a SGNE if and only if the following KKT conditions are satisfied \cite[Th. 4.6]{facchinei2010}: for all $i \in \mc I$
\vspace{-.15cm}\begin{equation}\label{game_kkt}
\left\{\begin{array}{l}
0 \in \EE[\nabla_{x_i} J_i(x_i^{*}, \boldsymbol x_{-i}^{*},\xi)]+\mathrm{N}_{\Omega_i}\left(x_i^{*}\right)+A_i^{\top} \lambda_i \\
0\in -(A\boldsymbol x^*-b)+\mathrm N_{\RR^m_{\geq 0}}(\lambda^*).
\end{array}\right.
\vspace{-.15cm}\end{equation}
Similarly, we can use the KKT conditions to characterize a variational GNEP (v-GNEP), studying the Lagrangian function associated to the SVI. Since $\bs x^{*}$ is a solution of $\op{SVI}(\bs{\mc X} , \FF)$ if and only if $\bs x^{*} \in \operatorname{argmin}\limits_{\bs y \in \mc{X}}\left(\bs y-\bs x^{*}\right)^{\top} \FF\left(\bs x^{*}\right),$ the associated KKT optimality conditions are, for all $i \in \mc I$
\vspace{-.15cm}\begin{equation}\label{VI_KKT}
\left\{\begin{array}{l}
0 \in\EE[\nabla_{x_i} J_i(x_i^{*}, \boldsymbol x_{-i}^{*},\xi)]+\mathrm{N}_{\Omega_i}\left(x_i^{*}\right)+A_i^{\top} \lambda, \\ 
0\in -(A\boldsymbol x^*-b)+\mathrm N_{\RR^m_{\geq 0}}(\lambda^*).
\end{array}\right.
\vspace{-.15cm}\end{equation}
We note that (\ref{VI_KKT}) can be written in compact form as\vspace{-.15cm}
$$
0\in\mc T(\bs x,\bs \lambda):=\left[\begin{array}{c}
\mathrm{N}_{\Omega}(\boldsymbol x)+\FF(\boldsymbol x)+A^{\top} \lambda \\ 
\mathrm{N}_{\RR_{ \geq 0}^{m}}(\lambda)-(A \boldsymbol x-b)
\end{array}\right],\vspace{-.15cm}
$$
where $\mc T:\bs{\mc X}\times \RR^m_{\geq 0}\rightrightarrows \RR^{n}\times\RR^m$ is a set-valued mapping.
It follows that the v-GNE correspond to the zeros of the mapping $\mc T$. The next proposition shows the relation between SGNE and variational equilibria.

\begin{lemma}\label{var_eq}\cite[Th. 3.1]{facchinei2007vi}
The following statements hold:
\begin{enumerate}
\item Let $\boldsymbol x^*$ be a solution of $\SVI (\mc X, \FF)$ at which the KKT conditions (\ref{VI_KKT}) hold. Then $\boldsymbol x^*$ is a solution of the SGNEP at which the KKT conditions (\ref{game_kkt}) hold with
$\lambda_1=\lambda_2=\dots=\lambda_N=\lambda^*$
\item Viceversa, let $\boldsymbol x^*$ be a solution of the SGNEP at which KKT conditions (\ref{game_kkt}) hold with $\lambda_1=\lambda_2=\dots=\lambda_N=\lambda^*$. Then, $\boldsymbol x^*$ is a solution of $\op{SVI}(\mc X , \FF)$.\hfill\small$\blacksquare$
\end{enumerate}
\end{lemma}

Essentially, Lemma \ref{var_eq} says that variational equilibria are those such that the shared constraints have the same dual variable for all the agents.

\section{Preconditioned forward--backward generalized Nash equilibrium seeking}
In this section, we propose a distributed forward--backward algorithm for finding variational equilibria of the game in \eqref{game}.

We suppose that each agent $i$ only knows its local data, i.e., $\Omega_i$, $A_i$ and $b_i$. Moreover, each player is able to compute, given $\op{col}(x_i,\boldsymbol x_{-i})$, $\EE[\nabla_{x_i} J_i(x_i,\boldsymbol x_{-i},\xi)]$ (or an approximation, as exploited later in the section). We assume therefore that each agent has access to the action variables that affect its local gradient (full information setup). These information can be retrieved, for each agent $i$, from the set $\mc N_i^J$, that is, the set of agents $j$ whose action $x_j$ explicitly influences $\JJ_i$.

%We assume that the agents are connected on the interference graph through which they have access to the information necessary for the cost function. Formally, the interference graph $\mc G_J=(\mc I,\mc E_J)$ defines the relationship between the cost function of the agents. An edge $(i,j)\in\mc E_J$ if $\JJ_i(x_i,x_{-i},\xi)$ explicitly depends on the decision variable $x_j$ of agent $j$ \cite{yi2019,yu2017}.

Since, by Lemma \ref{var_eq}, the configuration of the v-GNE requires consensus of the dual variables, we introduce an auxiliary variable $z_i\in\RR^m$ for all $i\in\mc I$. The role of $\bs z=\op{col}(z_1,\dots,z_N)$ is to enforce consensus but it does not affect the property of the operators and of the algorithm. More details on this variable are given in Section \ref{sec_pFB}. The auxiliary variable and a local copy of the dual variable $\lambda_i$ are shared through the dual variables graph, $\mc G_\lambda=(\mc I,\mc E_\lambda)$. The set of edges $\mc E_\lambda$ represents the exchange of the private information on the dual variables: $(i,j)\in\mc E_\lambda$ if player $i$ can receive $\{\lambda_j,z_j\}$ from player $j$. The set of neighbours of $i$ in $\mc G_\lambda$ is given by $\mathcal{N}_{i}^{\lambda}=\left\{j |(j, i) \in \mathcal{E}_{\lambda}\right\}$ \cite{yi2019,yu2017}. \smallskip%In conclusion, each agent control his own decision variable and a local copy of the dual variable $\lambda_i$ and of the auxiliary variable $z_i$ while the information on the other players are shared through the graphs. 
\begin{standassumption}\label{ass_graph}
The dual variable graph is undirected and connected.\fineass
\end{standassumption}
Each agent feasible set implicitly depends on all the other agents variables through the shared constraints. Thus, to reach consensus of the multipliers, all agents must coordinate and therefore, $\mc G_\lambda$ must be connected.

The weighted adjacency matrix of the dual variables graph is indicated with $W\in\RR^{N\times N}$. 
Let $L=D-W\in\RR^{N\times N}$ be the Laplacian matrix associated to the adjacency matrix $W$, where $D=\op{diag}(d_1,\dots,d_N)$ is the diagonal matrix of the degrees and $d_i=\sum_{j=1}^{N} w_{i,j}$. It follows from Standing Assumption \ref{ass_graph} that $W$ and the associated Laplacian $L$ are both symmetric, i.e., $W=W^\top$ and $L = L^\top$. Moreover, Standing Assumption \ref{ass_graph} is fundamental to guarantee that the coupling constraints are satisfied since agents have to reach consensus of the dual variables. 

Next we present a distributed forward--backward algorithm with damping for solving the SGNEP in \eqref{game} (Algorithm \ref{DSFB_ave}). For each agent $i$, the variables $x_i^k$, $z_i^k$ and $\lambda_i^k$ denote the local variables $x_i$, $z_i$ and $\lambda_i$ at the iteration time $k$ while $\alpha_i$, $\nu_i$ and $\sigma_i$ are the step sizes. 

\begin{algorithm}
\caption{Distributed Stochastic Forward--Backward}\label{DSFB_ave}
Initialization: $x_i^0 \in \Omega_i, \lambda_i^0 \in \RR_{\geq0}^{m},$ and $z_i^0 \in \RR^{m} .$\\
Iteration $k$: Agent $i$\\
(1): Receives $x_{j, k}$ for $j \in \mathcal{N}_{i}^{h}, \lambda_{j}^k$ for $j \in \mathcal{N}_{i}^{\lambda}$ then updates:
$$\begin{aligned}
&\tilde x_i^k=\op{proj}_{\Omega_{i}}[x_i^k-\alpha_{i}(\hat F_{i}(x_i^k, \boldsymbol{x}_{-i}^k,\xi_i^k)-A_{i}^\top \lambda_i^k)]\\
&\tilde z_i^k=z_i^k+v_{i} \sum\nolimits_{j \in \mathcal{N}_{i}^{\lambda}} w_{i,j}(\lambda_i^k-\lambda_{j}^k)\\
\end{aligned}$$
(2): Receives $z_{j, k+1}$ for all $j \in \mathcal{N}_{i}^{\lambda}$ then updates:
$$\begin{aligned}
&\tilde \lambda_i^k=\op{proj}_{\RR_{+}^{m}}\left[\lambda_i^k+\sigma_{i}\left(A_{i}(2 \tilde x_i^k-x_i^k)-b_{i}\right)\right.\\
&\qquad+\sigma_{i}\sum\nolimits_{j \in \mathcal{N}_{i}^{\lambda}} w_{i,j}\left(2(\tilde z_i^k-\tilde z_{j}^k)-(z_i^k-z_{j}^k)\right)\\
&\qquad-\sigma_{i}\sum\nolimits_{j \in \mathcal{N}_{i}^{\lambda}} \left.w_{i,j}(\lambda_i^k-\lambda_{j}^k)\right]\\
&x_{i}^{k+1}=(1-\delta)x_i^k+\delta\tilde x_i^k\\
&z_{i}^{k+1}=(1-\delta)z_i^k+\delta\tilde z_i^k\\
&\lambda_{i}^{k+1}=(1-\delta)\lambda_i^k+\delta\tilde \lambda_i^k\\
\end{aligned}$$
\end{algorithm}

Since the distribution of the random variable is unknown, in the algorithm we have replaced the expected value with a sample average approximation (SAA). 
We assume to have access to a pool of i.i.d. sample realizations of the random variable collected, for all $k\in\NN$ and for each agent $i\in\mc I$, in the vectors $\xi_i^k$. At each time $k$, we have\vspace{-.15cm}
\vspace{-.15cm}\begin{equation}\label{approx_SAA}
\hat F_i(x_i^k,\boldsymbol x_{-i}^k,\xi_i^k):=\frac{1}{N_k}\sum_{s=1}^{N_k}\nabla_{x_i} J_i(x_i^k,\boldsymbol x_{-i}^k,\xi_i^{(s)})
\vspace{-.15cm}\end{equation}
where $N_k$ is the batch size. As usual in SAA, we assume that the batch size increases over time according to the following lower bound. 
\begin{assumption}\label{ass_batch}
There exist $c,k_0,a>0$ such that, for all $k\in\NN$, $N_k\geq c(k+k_0)^{a+1}$.
\fineass\end{assumption}
We define the distance of the expected value and its approximation as $\epsilon^k=\hat F(\boldsymbol x^k,\xi^k)-\FF(\boldsymbol x^k)$.\\
Let us define the filtration $\mc F=\{\mc F_k\}$, that is, a family of $\sigma$-algebras such that $\mathcal{F}_{0} = \sigma\left(X_{0}\right)$, for all $k \geq 1$, 
$\mathcal{F}_{k} = \sigma\left(X_{0}, \xi_{1}, \xi_{2}, \ldots, \xi_{k}\right)$
and $\mc F_k\subseteq\mc F_{k+1}$ for all $k\geq0$. $\mc F$ collects the informations that each agent has at the beginning of iteration $k$. We note that the process $\epsilon_k$ is adapted to $\mc F$ and, therefore, satisfies the following assumption.
\begin{assumption}
For al $k\geq 0$, $\EEk{\epsilon_k}=0$ a.s.
\fineass\end{assumption}
The stochastic error has a vanishing second moment.
\begin{assumption}\label{ass_error}
For all $k$ and $C>0$, the stochastic error is such that $\EE[\|\epsilon^k\||\mc F_k]\leq C\sigma^2/N_k$.\fineass
\end{assumption}
Such a bound for the stochastic error can be obtained as a consequence of some milder assumptions that lay outside the scope of this work. We refer to \cite[Lem. 4.2]{staudigl2019}, \cite[Lem. 3.12]{iusem2017} for more details.

Furthermore, we assume that the step sizes are small enough as formalized next.
\begin{standassumption}\label{ass_phi}
Let $\delta\in(0,1]$. Let $d_i=\sum_{j=1}^{N} w_{i,j}$ and define $d^*=\max_{i\in \mc N}\{d_i\}$. Let 
%\vspace{-.15cm}\begin{equation}\label{beta}
$0<\beta \leq\min \left\{\frac{1}{2 d^{*}}, \frac{\eta}{\ell^{2}}\right\},$
%\vspace{-.15cm}\end{equation}
where $\eta$ and $\ell$ are respectively the strongly monotone and the Lipschitz constants as in Standing Assumption \ref{ass_F}. The parameter $\tau$ is positive and $\tau<\frac{1}{2\beta}$. The step sizes $\alpha$, $\nu$ and $\sigma$ satisfy, for any agent $i$ 
\vspace{-.15cm}\begin{equation}\label{parameters_phi}
\begin{aligned}
0&<\alpha_{i} \leq\left(\max _{j=1, \ldots, n_{i}}\{\sum\nolimits_{k=1}^{m}|[A_{i}^{T}]_{j k}|\}+\tau\right)^{-1} \\ 
0&<\nu_{i} \leq\left(2 d_{i}+\tau\right)^{-1}\\
0&<\sigma_{i} \leq\left(\max _{j=1, \ldots, m}\{\sum\nolimits_{k=1}^{k_{i}}|[A_{i}]_{j k}|\}+2 d_{i}+\tau\right)^{-1}\\
\end{aligned}
\vspace{-.15cm}\end{equation}
where $[A_i^\top]_{jk}$ indicates the entry $(j,k)$ of the matrix $A_i^\top$. \fineass

%The parameter $\tau$ is positive and $\tau<\frac{1}{2\beta}$ with
%\vspace{-.15cm}\begin{equation}\label{beta}
%0<\beta \leq\min \left\{\frac{1}{2 d^{*}}, \frac{\eta}{\ell^{2}}\right\}
%\vspace{-.15cm}\end{equation}
%where $d^*=\max_{i\in \mc N}\{d_i\}$ and $\eta$ and $\ell$ are respectively the strongly monotone and the Lipschitz constants.\fineass

\end{standassumption}
%An insight on these conditions is given in Section \ref{sec:conv}. 
Specifically, the choice of $\beta$ will be clear after Lemma \ref{propertiesAB}. Besides the fact that the constant $\beta$ can be hard to compute, we note that the step sizes can be taken constant and each agent can independently choose its own.

We are now ready to state our convergence result.
\begin{theorem}\label{theo_conv}
Let Assumptions \ref{ass_batch} and \ref{ass_error} hold. Then, the sequence generated by Algorithm \ref{DSFB_ave} with $\hat F_i$ as in (\ref{approx_SAA}) for all $i\in\mc N$ converges a.s. to a v-GNE of the game in (\ref{game}).
\end{theorem}
\begin{proof}
See Section \ref{sec:conv}. 
\end{proof}

\section{Convergence analysis}\label{sec:conv}
\subsection{Preconditioned forward backward operator splitting}\label{sec_pFB}
In this section, we prove convergence of Algorithm \ref{DSFB_ave}. %Let Assumptions \ref{ass_batch} - \ref{ass_phi} hold. %We start with the properties of the operators $\Phi^{-1}\bar{\mc A}$ and $\Phi^{-1}\bar{\mc B}$. }

We first note that the mapping $\mc T$ can be written as the sum of two operators. Specifically, $\mc T = \mc A + \mc B$ where
\vspace{-.15cm}\begin{equation}\label{op}
\begin{aligned}
\mc{A} &:\left[\begin{array}{l}
\boldsymbol x \\
\boldsymbol\lambda
\end{array}\right] \hspace{-.1cm}\mapsto\hspace{-.1cm}\left[\begin{array}{c}
\FF(\boldsymbol x) \\ 
b
\end{array}\right]\\
\mc{B} &:\left[\begin{array}{l}
\boldsymbol x \\ 
\boldsymbol\lambda
\end{array}\right] \hspace{-.1cm}\mapsto\hspace{-.1cm}\left[\begin{array}{c}
\op{N}_{\Omega}(\boldsymbol x) \\ 
\op{N}_{\RR_{ \geq 0}^{m}}(\lambda)
\end{array}\right]+\left[\begin{array}{cc}
0 & A^{\top} \\ 
-A & 0
\end{array}\right]\left[\begin{array}{l}
\boldsymbol x \\ 
\boldsymbol\lambda
\end{array}\right].
\end{aligned}
\vspace{-.15cm}\end{equation}
%The formulation $\mc T = \mc A + \mc B$ is called splitting of $\mc T$ \cite{bau2011}. 
We note that finding a solution of the variational SGNEP translates in finding $x^*\in\op{zer}(\mc A+\mc B)$. 

Let $L$ be the Laplacian matrix of $\mc G_\lambda$  and set ${\bf{L}}=L\otimes \op{I}_m\in\RR^{Nm\times Nm}$.
To impose consensus on the dual variables, the authors in \cite{yi2019} proposed the Laplacian constraint ${\bf{L}}\bs\lambda=0$. Then, to preserve monotonicity one can augment the two operators $\mc A$ and $\mc B$ introducing the auxiliary variable $\bs z$. Define ${\bf{A}}=\op{diag}\{A_1,\dots,A_N\}\in\RR^{Nm\times n}$ and $\boldsymbol\lambda=\op{col}(\lambda_1,\dots,\lambda_N)\in\RR^{Nm}$ and similarly let us define $\bar b$ of suitable dimensions.
Then, we introduce
\vspace{-.15cm}\begin{equation}\label{extended_op}
\begin{aligned}
\bar{\mc{A}} &:\left[\hspace{-.1cm}\begin{array}{l}
\boldsymbol x \\
\boldsymbol z\\
\boldsymbol \lambda
\end{array}\hspace{-.1cm}\right] \hspace{-.1cm}\mapsto\hspace{-.1cm}\left[\begin{array}{c}
\FF(\boldsymbol x) \\ 
0\\
\bar b
\end{array}\right]+\left[\begin{array}{c}
0\\
0\\
{\bf{L}}\boldsymbol\lambda
\end{array}\right]\\
\bar{\mc{B}} &:\left[\hspace{-.1cm}\begin{array}{l}
\boldsymbol x \\ 
\boldsymbol z\\
\boldsymbol \lambda
\end{array}\hspace{-.1cm}\right] \hspace{-.1cm}\mapsto\hspace{-.1cm}\left[\begin{array}{c}
\op{N}_{\Omega}(\boldsymbol x) \\ 
{\bf{0}}_{mN}\\
\hspace{-.1cm}\op{N}_{\RR_{ \geq 0}^{m}}(\boldsymbol \lambda)\hspace{-.1cm}
\end{array}\right]\hspace{-.1cm}+\hspace{-.1cm}\left[\hspace{-.1cm}\begin{array}{ccc}
0 & 0 & {\bf{A}}^{\top} \\ 
0 & 0 & {\bf{L}}\\
-{\bf{A}} & -{\bf{L}} & 0
\end{array}\hspace{-.1cm}\right]\hspace{-.2cm}\left[\hspace{-.1cm}\begin{array}{l}
\boldsymbol x \\ 
\boldsymbol z\\
\boldsymbol \lambda
\end{array}\hspace{-.1cm}\right]\hspace{-.1cm}.
\end{aligned}
\vspace{-.15cm}\end{equation}
From now on, let us indicate $\boldsymbol\omega=\op{col}(\boldsymbol x,\boldsymbol z,\boldsymbol \lambda)$.
The operators $\bar{\mc A}$ and $\bar{\mc B}$ in \eqref{extended_op} have the following properties.

\begin{lemma}\label{propertiesAB}
The operator $\bar{\mc B}$ is maximally monotone and $\bar{\mc A}$ is $\beta$-cocoercive with $\beta$ as in Standing Assumption \ref{ass_phi}, i.e., $\langle \bar{\mc A}(x)-\bar{\mc A}(y),x-y\rangle \geq \beta\|\bar{\mc A}(x)-\bar{\mc A}(y)\|^{2}.$
\end{lemma}
\begin{proof}
It follows from \cite[Lem. 5]{yi2019}.
\end{proof}
%\begin{proof}
%Set $\bar {\mc B}=\mc B_1+\mc B_2$. The normal cone is maximally monotone \cite[Ex 20.26]{bau2011} and so is the vector ${\bf{0}}_{mN}$. Then, $\mc B_1$ is maximally monotone. The second addend is a skew symmetric matrix and $\op{dom}\mc B_2=\RR^{n+2mN}$, therefore $\mc B_2$ is maximally monotone \cite[Ex 20.35]{bau2011}. It follows that $\bar{\mc B}$ is maximally monotone \cite[Proposition 20.23]{bau2011}. \\
%The operator $\bar{\mc A}$ is cocoercive since it holds that
%$$\begin{aligned}
%\langle \FF(\boldsymbol x_{1})-\FF(\boldsymbol x_{2}), \boldsymbol x_{1}-\boldsymbol x_{2}\rangle&\geq\eta\normsq{\boldsymbol x_1-\boldsymbol x_2}\\
%&\geq\frac{\eta}{\ell^2}\normsq{\FF(\boldsymbol x_{1})-\FF(\boldsymbol x_{2})}.\\
%\end{aligned}$$
%Moreover, since the eigenvalue of ${\bf{L}}$ are $\op{col}(0,s_2,\dots,s_N)\otimes \mathbf{1}_m$, ${\bf{L}}\boldsymbol\lambda+\bar b$ is $\norm{L}$-Lipschitz continuous and by the Baillon-Haddad Theorem 
%$$\langle{\bf{L}} \boldsymbol \lambda_{1}-{\bf{L}} \boldsymbol \lambda_{2}, \boldsymbol \lambda_{1}-\boldsymbol \lambda_{2}\rangle \geq \frac{1}{\norm{L}}\normsq{{\bf{L}} \boldsymbol \lambda_{1}-{\bf{L}} \boldsymbol \lambda_{2}}.$$
%Since $\norm{L}\leq s_N$ and $d^*\leq s_N\leq 2d^*$, $\norm{L}\leq2d^*$. In conclusion taking $\beta\in(0,\min\{1/2d^*,\eta/\ell^2\}]$ we have that $\bar {\mc A}$ is $\beta$-cocoercive.
%\end{proof}
Notice that this result explains the choice of the parameter $\beta$ in Assumption \ref{ass_phi}.
The following lemma shows that the points $\bs\omega\in\op{zer}(\bar{\mc A}+\bar{\mc B})$ provide a v-SGNE.

\begin{lemma}\label{lemma_zeri}
Consider the operators $\bar{\mc A}$ and $\bar{\mc B}$ in (\ref{extended_op}), and the operators $\mc A$ and $\mc B$ in (\ref{op}). Then the following hold.
\begin{itemize}
\item[(i)] Given any $\boldsymbol\omega^* \in \op{zer}(\bar{\mc A}+\bar{\mc B})$, $\boldsymbol{x}^{*}$ is a v-GNE of game in (\ref{game}), i.e., $\boldsymbol x^*$ solves the $\SVI(\mc X, \FF)$ in (\ref{eq_vi}). Moreover $\boldsymbol \lambda^{*}=\mathbf{1}_{N} \otimes \lambda^{*},$ and $(\boldsymbol x^*,\lambda^{*})$ satisfy the KKT condition in (\ref{VI_KKT}) i.e., $\op{col}(\boldsymbol x^*, \lambda^{*}) \in \op{zer}(\mc A+\mc B)$
\item[(ii)] $\op{zer}(\mc A+\mc B) \neq \emptyset$ and $\op{zer}(\bar{\mc A}+\bar{\mc B}) \neq \emptyset$
\end{itemize}
\end{lemma}
\begin{proof}
It follows from \cite[Th. 2]{yi2019}.
%The result can be proved similarly to \cite[Th. 2]{yi2019}. Writing explicitly $\boldsymbol\omega^*\in\op{zer}(\bar{\mc A}+\bar{\mc B})$ and using properties of the Laplacian and of the normal cone, we obtain exactly the KKT condition in \eqref{VI_KKT}. Then, by Lemma \ref{var_eq}, $(\boldsymbol x^*,\lambda^*)$ satisfy \eqref{game_kkt}, $(\boldsymbol x^*,\lambda^*)\in\op{zer}(\mc A+\mc B)$ and $x^*$ is an equilibrium of the game in \eqref{game}. The converse follows showing that there exist an auxiliary variable $z$, given $(\boldsymbol x^*,\lambda^*)\in\op{zer}(\mc A+\mc B)$, using the properties of the normal cone of the dual variable.
\end{proof}

Unfortunately, the operator $\bar{\mc B}$ is monotone but not cocoercive, due to the skew symmetric matrix, hence, we cannot directly apply the FB operator splitting \cite[\S 26.5]{bau2011}. To overcome this issue, the authors in \cite{yi2019} introduced a preconditioning matrix $\Phi$. Thanks to $\Phi$, the zeros of the mapping $\mc T$ correspond to the fixed point of a specific operator that depends on the operators $\bar{\mc A}$ and $\bar{\mc B}$ as exploited in \cite{belgioioso2018,yi2019}. Indeed, it holds that, for any matrix $\Phi\succ0$, $\boldsymbol\omega\in\op{zer}(\bar{\mc A}+\bar{\mc B})$ if and only if, 
\vspace{-.15cm}\begin{equation}\label{FB_fix}
\boldsymbol\omega=\op{J}_{\Phi^{-1}\bar{\mc B}}(\op{Id}-\Phi^{-1}\bar{\mc A})(\boldsymbol\omega)
\vspace{-.15cm}\end{equation}
where $\op{J}_{\Phi^{-1}\bar{\mc B}}=(\mathrm{Id}+\Phi^{-1} \bar{\mc B})^{-1}$ is the resolvent of $\bar{\mc B}$ and represent the backward step and $(\mathrm{Id}-\Phi^{-1} \bar{\mc A})$ is the (stochastic) forward step.
In the deterministic case, convergence of the FB splitting is guaranteed by \cite[Section 26.5]{bau2011}. In the stochastic case, the FB algorithm, as it is in (\ref{FB_fix}), is known to converge for strongly monotone mappings \cite{xu2013}. For this reason, we focus on the following damped FB algorithm 
\vspace{-.15cm}\begin{equation}\label{FB_ave}
\begin{cases}
\tilde{\boldsymbol\omega}_k=\op{J}_{\Phi^{-1}\bar{\mc B}}(\mathrm{Id}-\Phi^{-1} \bar {\mc A})(\boldsymbol\omega_k)\\
\boldsymbol\omega_{k+1}=(1-\delta_k)\boldsymbol\omega_k+\delta_k \tilde{\boldsymbol\omega}_k
\end{cases}
\vspace{-.15cm}\end{equation}
that converges with cocoercivity of $\Phi^{-1}\bar{\mc A}$ \cite{rosasco2016}. We show that this is true in Lemma \ref{op_phi}. 

First, we show that (\ref{FB_ave}) is equivalent to Algorithm \ref{DSFB_ave}. 
%As spoiled in the introduction, in the stochastic case, we need to take an approximation of the expected value. At this stage, it is not important if we use sample average or stochastic approximation, therefore, in what follows, we replace $\bar{\mc A}$ with 
%$$\hat{\mc A}:\left[\begin{array}{c}
%\bf{x} \\
%z\\
%\lambda
%\end{array}\right] \mapsto\left[\begin{array}{c}
%\hat J(\boldsymbol x,\xi) \\ 
%0\\
%b
%\end{array}\right]$$
%where $\tilde F$ is an approximation of the expected value mapping $F$ in (\ref{eq_sub_ee}) given the random vector $\xi$. 
Note that, if we write the resolvent explicitly, the first step of Equation (\ref{FB_ave}) can be rewritten as
\vspace{-.15cm}\begin{equation}\label{inclusion}
-\mc{\bar A}(\boldsymbol\omega^{k}) \in \bar{\mc{B}}(\tilde{\boldsymbol\omega}^{k})+\Phi(\tilde{\boldsymbol\omega}^{k}-\boldsymbol\omega^{k}).
\vspace{-.15cm}\end{equation}
The matrix $\Phi$ should be symmetric, positive definite and such that $\boldsymbol\omega_k$ is easy to be computed \cite{belgioioso2018}. 

We define $\alpha^{-1}=\op{diag}\{\alpha_1^{-1}\op{I}_{n_1},\dots,\alpha_N^{-1}\op{I}_{n_N}\}\in\RR^{n\times n}$ and similarly $\sigma^{-1}$ and $\nu^{-1}$ of suitable dimensions. Let
\vspace{-.15cm}\begin{equation}\label{phi}
\Phi=\left[\begin{array}{ccc}
\alpha^{-1} & 0 & -{\bf{A}}^\top\\
0 & \nu^{-1} & -{\bf{L}}\\
-{\bf{A}} & -{\bf{L}} & \sigma^{-1}
\end{array}\right]
\vspace{-.15cm}\end{equation}
and suppose that the parameters $\alpha_i$, $\nu_i$ and $\sigma_i$ satisfy \eqref{parameters_phi} in Standing Assumption \ref{ass_phi}. 
%Since it has to be positive definite, we assume that the parameters $\alpha$, $\nu$ and $\sigma$ satisfy, for any argent $i$ and $\tau>0$ 
%\vspace{-.15cm}\begin{equation}\label{parameters_phi}
%\begin{aligned}
%0&<\alpha_{i} \leq\left(\max _{j=1, \ldots, n_{i}}\{\sum\nolimits_{k=1}^{m}|[A_{i}^{T}]_{j k}|\}+\tau\right)^{-1} \\ 
%0&<\nu_{i} \leq\left(2 d_{i}+\tau\right)^{-1}\\
%0&<\sigma_{i} \leq\left(\max _{j=1, \ldots, m}\{\sum\nolimits_{k=1}^{k_{i}}|[A_{i}]_{j k}|\}+2 d_{i}+\tau\right)^{-1}\\
%\end{aligned}
%\vspace{-.15cm}\end{equation}
%where $[A_i^\top]_{jk}$ indicates the entry $(j,k)$ of the matrix $A_i^\top$ and $d_i$ is the weighted degree of agent $i$ in $\mc G_\lambda$. 

We can obtain conditions \eqref{parameters_phi} imposing $\Phi$ to be diagonally dominant. This, in combination with the fact that it is symmetric, implies that $\Phi$ is also positive definite. Then, the operators $\Phi^{-1}\bar{\mc A}$ and $\Phi^{-1}\bar{\mc B}$ satisfy the following properties under the $\Phi$-induced norm $\norm{\cdot}_\Phi$.
\begin{lemma}\label{op_phi}
Given $\bar{\mc A}$ and $\bar{\mc B}$ in (\ref{extended_op}) and $\beta$ and $\tau$ as in Standing Assumption \ref{ass_phi}, the following statements hold:
\begin{itemize}
\item[(i)] $\Phi^{-1}\bar{\mc A}$ is $\beta\tau$-cocoercive;
\item[(ii)] $\Phi^{-1}\bar{\mc B}$ is maximally monotone. 
\end{itemize}
\end{lemma}
\begin{proof}
It follows from \cite[Lem. 7]{yi2019}.
\end{proof}

\subsection{Stochastic sample average approximation}
Since the expected value can be hard to compute, we need to take an approximation. At this stage, it is not important if we use sample average or stochastic approximation, therefore, in what follows, we replace $\bar{\mc A}$ with \vspace{-.15cm}
$$\hat{\mc A}:\left[\begin{array}{c}
\boldsymbol x \\
\boldsymbol z\\
\boldsymbol\lambda
\end{array}\right] \mapsto\left[\begin{array}{c}
\hat F(\boldsymbol x,\xi) \\ 
0\\
\bar b
\end{array}\right]+\left[\begin{array}{c}
0\\
0\\
{\bf{L}}\boldsymbol\lambda
\end{array}\right]\vspace{-.15cm}$$
where $\hat F$ is an approximation of the expected value mapping $\FF$ in (\ref{map_vi}) given a vector sample of the random variable $\xi$. 
Then, (\ref{inclusion}) can be rewritten as 
\vspace{-.15cm}\begin{equation}\label{inclusion_extended}
\begin{aligned}
-\left[\begin{array}{c}
\hat F(\boldsymbol x,\xi)\\
0\\
{\bf{L}}\boldsymbol\lambda+\bar b
\end{array}\right]\hspace{-.1cm}\in&\hspace{-.1cm}\left[\begin{array}{ccc}
0 & 0 & {\bf{A}}^{\top} \\ 
0 & 0 & {\bf{L}}\\
-{\bf{A}} & -{\bf{L}} & 0
\end{array}\right]\hspace{-.2cm}\left[\begin{array}{c}
\tilde {\boldsymbol x}_{k}\\
\tilde {\boldsymbol z}_{k}\\
\tilde {\boldsymbol\lambda}_{k}
\end{array}\right]\\
+&\hspace{-.1cm}\left[\hspace{-.1cm}\begin{array}{ccc}
\alpha^{-1} & 0 & -{\bf{A}}^\top\\
0 & \nu^{-1} & -{\bf{L}}\\
-{\bf{A}} & -{\bf{L}} & \sigma^{-1}
\end{array}\hspace{-.1cm}\right]\hspace{-.2cm}\left[\begin{array}{c}
\tilde {\boldsymbol x}_{k}-\boldsymbol x_k\\
\tilde {\boldsymbol z}_{k}-\boldsymbol z_k\\
\tilde {\boldsymbol\lambda}_{k}-\boldsymbol \lambda_k
\end{array}\right]
\end{aligned}
\vspace{-.15cm}\end{equation}
By expanding (\ref{inclusion_extended}), we obtain the first steps of Algorithm \ref{DSFB_ave}. The damping part is distributed and it does not need preconditioning.

We note that, thanks to the fact that $\Phi+\bar{\mc B}$ is lower block triangular, the iterations of Algorithm \ref{DSFB_ave} are sequential, that is, $\tilde {\boldsymbol\lambda}_{k}$ use the last update $\tilde {\boldsymbol x}_{k}$ and $\tilde {\boldsymbol z}_{k}$ of the agents strategies and of the auxiliary variable respectively.

%
%\section{Convergence Proofs}\label{sec:proofs}
%
%In this section, we prove the convergence of Algorithm \ref{DSFB_ave}. We start with some preliminary results on the properties of the operators $\Phi^{-1}\bar{\mc A}$ and $\Phi^{-1}\bar{\mc B}$.

%Consider a matrix $\Phi$ as in (\ref{phi}) and suppose that the parameters $\alpha_i$, $\nu_i$ and $\sigma_i$ satisfy \eqref{parameters_phi}. 

%We can obtain Conditions \eqref{parameters_phi} imposing $\Phi$ to be diagonally dominant. This, in combination with the fact that it is symmetric, implies that $\Phi$ is also positive definite. Then, the operators $\Phi^{-1}\bar{\mc A}$ and $\Phi^{-1}\bar{\mc B}$ satisfy the following properties under the $\Phi$-induced norm $\norm{\cdot}_\Phi$.
%\begin{lemma}\label{op_phi}
%Suppose $\alpha_i$, $\nu_i$ and $\sigma_i$ satisfy (\ref{parameters_phi}) for all $i\in\mc N$. The following statements holds.
%\begin{itemize}
%\item $\Phi^{-1}\bar{\mc A}$ is $\beta\tau$-cocoercive where $0<\beta \leq\min \left\{\frac{1}{2 d^{*}}, \frac{\eta}{\ell^{2}}\right\}$, $\tau<\frac{1}{2\beta}$ and $d$ is the maximum weighted degree of $\mc G_\lambda$, $\eta$ is the strongly monotone constant and $\ell$ is the Lipschitz constant.
%\item $\Phi^{-1}\bar{\mc B}$ is maximally monotone. \fineass
%\end{itemize}
%\end{lemma}
%\begin{proof}
%Cocoercivity of $\Phi^{-1}\bar{\mc A}$ follows by the cocoercivity of $\bar{\mc A}$ and $\Phi^{-1}\bar{\mc B}$ is maximally monotone because $\bar{\mc B}$ is maximally monotone. See \cite[Lemma 7]{yi2019} for more details.
%\end{proof}

We are now ready to prove our convergence result.

\begin{proof}[Proof of Theorem \ref{theo_conv}]
The iterations of Algorithm \ref{DSFB_ave} are obtained by expanding (\ref{FB_ave}), solving for $\tilde x_{k}$, $\tilde z_{k}$ and $\tilde \lambda_{k}$ and adding the damping iteration. Therefore, Algorithm \ref{DSFB_ave} is the FB iteration with damping as in (\ref{FB_ave}). The convergence of the sequence $(\boldsymbol x^k,\boldsymbol\lambda^k)$ to a v-GNE of the game in \eqref{game} then follows by \cite[Th. 3.2]{rosasco2016} and Lemma \ref{lemma_zeri} since $\Phi^{-1}\bar{\mc A}$ is cocoercive by Lemma \ref{propertiesAB}.
\end{proof}

\subsection{Discussion}
The original result in \cite{rosasco2016} shows convergence for cocoercive and uniformly monotone operators.
%The original convergence result in \cite{rosasco2016} is for cocoercive and uniformly monotone operators. 
Moreover, they provide the proof for a generic approximation of the random mapping $\FF$. We note that fixing the type of approximation (SAA or SA) can be important for weakening the assumptions. Indeed, using the SAA scheme, cocoercivity is enough for convergence without further monotonicity assumptions. On the other hand, the SA approach requires cocoercivity and \textit{strict} monotonicity.  
Unfortunately, the mapping $\Phi^{-1}\bar{\mc A}$ is not strictly monotone (due to the presence of the Laplacian matrix), therefore we use SAA as in \eqref{approx_SAA}.

Concerning the stochastic error $\epsilon^k$, the assumption in \cite{rosasco2016} is similar to the so-called "variance reduction". Such an assumption is fundamental in the SAA scheme, but it can be avoided in the SA method \cite{koshal2013}. Indeed, in the SAA scheme, taking the average over a huge number of samples helps controlling the stochastic error and therefore finding a solution \cite{iusem2017,iusem2019}. For this reason, in our case the stepsize can be taken constant. For the SA scheme instead, the error is controlled in combination with the parameters involved, for instance, using a vanishing stepsize (possibly square summable) \cite{kannan2014} or using smoothing techniques (as a Tikhonov regularization) \cite{koshal2013}. In both cases, the damping parameter $\delta$ can be taken constant.

\section{Case study and numerical simulations}

As an example, we borrow an electricity market problem from \cite{xu2013} which can also be casted as a network Cournot game with markets capacity constraints \cite{yi2019,yu2017}.

Consider a set of $N$ generators (companies) that operate over a set of $m$ locations (markets). The random variable $\xi$ represent the demand uncertainty. Each generator decides the quantity $x_i$ of product to deliver in the $n_i$ markets it is connected with. Each company has a local cost function $c_i(x_i)$ related to the production of electricity and therefore we suppose it is deterministic.
%The cost function is deterministic as we suppose that the generators are able to compute the cost for production without uncertainty. On the other hand, 
Each market has a bounded capacity $b_j$ so that the collective constraints are given by $A\boldsymbol x\leq b$ where $A=[A_1,\dots,A_N]$ and $A_i$ specifies which market company $i$ participates in. Each location has a price, collected in $P:\RR^m\times \Xi\to\RR^m$ which is a linear function. The uncertainty variable can represent in this case the demand realization or some information about the past. %$P$ is supposed to be a linear function. 
The cost function of each agent is then given by $\JJ_i(x_i,x_{-i},\xi)=c_i(x_i)-\EE[P(\xi)^\top(A\boldsymbol x)A_ix_i].$
If $c_i(x_i)$ is strongly convex with Lipschitz continuous gradient and the prices are linear, the pseudo gradient of $\JJ_i$ is strongly monotone.% and Standing Assumption \ref{ass_cost} is satisfied.

\paragraph*{Numerical example}
As a numerical setting, we consider a set of 20 companies and 7 markets \cite{yi2019,yu2017}. Each company $i$ has a local constraint $0 \leq x_i \leq \gamma_i$ representing the capacity limit of generator $i$; each component of $\gamma_i$ is randomly drawn from $[1, 1.5]$.
% In terms of electricity market, this can be seen as the capacity limit of generator $i$. 
Each market $j$ has a maximal capacity $b_j$ randomly drawn from $[0.5, 1]$. The local cost function of the generator $i$ is $c_i(x_i) = \pi_i\sum_{j=1}^{n_i} ([x_i]_j)^2 + g_i^\top x_i$, where $[x_i]_j$ indicates the $j$ component of $x_i$.
$\pi_i$ is randomly drawn from $[1, 8]$, and each component of $g_i$ is randomly drawn from $[0.1, 0.6]$. Notice that $c_i(x_i)$ is strongly convex with Lipschitz continuous gradient. 

The price is taken as a linear function $P(\xi) = \bar P-D(\xi)A\boldsymbol x$ where each component of $\bar P =\op{col}(\bar P_1,\dots,\bar P_7)$ is randomly drawn from $[2,4]$. The uncertainty appears in the quantities $D(\xi)=\op{diag}\{d_1(\xi_1),\dots, d_7(\xi_7)\}$ that concern the total supply for each market. The entries of $D(\xi)$ are taken with a normal distribution with mean $0.8$ and finite variance. As in \cite{yi2019}, the dual variables graph is a cycle graph with the addition of the edges $(2,15)$ and $(6,13)$. The cost function of agent $i$ is influenced by the variables of the companies selling in the same market. This information can be retrieved from the graph in \cite[Fig. 1]{yi2019}.

The step sizes are the same for all agents $i$: $\alpha_i = 0.03$, $\nu_i = 0.2$, $\sigma_i= 0.03$. The initial point $x_{i,0}$ is randomly chosen within its local feasible set, and the initials dual variables $\lambda_{i,0}$ and $z_{i,0}$ are set to zero. The damping parameter is taken to be $\delta\in\{0.4, 0.7, 1\}$ to compare the results.

The plots in Fig. \ref{plot_sol} show the performance indicx $\frac{\norm{\boldsymbol x_{k+1}-\boldsymbol x^{*}}}{\norm{\boldsymbol x^{*}}}$ that indicates the convergence to a solution $x^*$. As one can see, the higher the averaging parameter the faster is the convergence. The oscillations depend on the approximation but we focus on the fact that the distance from the solution is decreasing. Moreover, it is interesting that for $\delta=1$ the algorithm is indeed a FB algorithm without damping and it converges with constant stepsize.
 
\begin{figure}
\centering
\includegraphics[scale=.4]{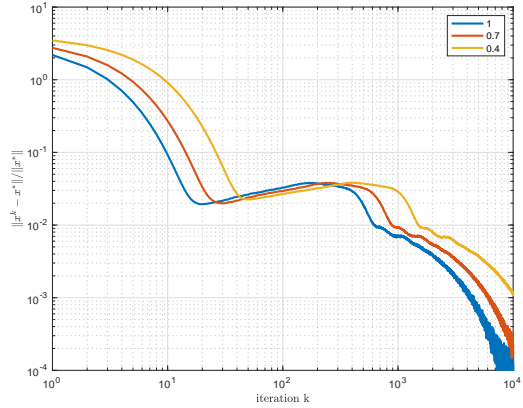}
\caption{Normalized distance from the GNE.}\label{plot_sol}
\end{figure}

\section{Conclusion}
The preconditioned forward--backward operator splitting is applicable to stochastic generalized Nash equilibrium problems to design distributed equilibrium seeking algorithms. Since the expected value is hard to compute in general, the sample average approximation can be used to ensure convergence almost surely.
Our simulations show that the damping step may be unnecessary for convergence. We will investigate this case as future research.

\bibliographystyle{IEEEtran}
\bibliography{IEEEabrv,Biblio}

\end{document}